\documentclass[11pt]{article}
\usepackage[margin=1in]{geometry}
\usepackage{amsmath,amsthm,amsfonts,amssymb}
\usepackage{graphicx,enumerate}
\graphicspath{ {./figures/} }
\usepackage[title]{appendix}

\usepackage[
            CJKbookmarks=true, 
            bookmarksnumbered=true,
            bookmarksopen=true,
            colorlinks=true,
            citecolor=red,
            linkcolor=blue,
            anchorcolor=red,
            urlcolor=blue
            ]{hyperref}

\newcommand{\Exp}{\mathbb{E}}

\usepackage{tikz}
\usepackage{tikz-qtree}

\usetikzlibrary{calc}
\usetikzlibrary{decorations.pathreplacing,decorations.markings}

\tikzstyle{dot}=[circle,fill,black,inner sep=1pt]

\tikzset{
  on each segment/.style={
    decorate,
    decoration={
      show path construction,
      moveto code={},
      lineto code={
        \path [#1]
        (\tikzinputsegmentfirst) -- (\tikzinputsegmentlast);
      },
      curveto code={
        \path [#1] (\tikzinputsegmentfirst)
        .. controls
        (\tikzinputsegmentsupporta) and (\tikzinputsegmentsupportb)
        ..
        (\tikzinputsegmentlast);
      },
      closepath code={
        \path [#1]
        (\tikzinputsegmentfirst) -- (\tikzinputsegmentlast);
      },
    },
  },
  mid arrow/.style={postaction={decorate,decoration={
        markings,
        mark=at position .5 with {\arrow[#1]{stealth}}
      }}},
  early arrow/.style={postaction={decorate,decoration={
        markings,
        mark=at position .2 with {\arrow[#1]{stealth}}
      }}},
}

\def\alternatecolorred{%
    \pgfkeysalso{red}%
    \global\let\alternatecolor\alternatecolorblue 
}
\def\alternatecolorblue{%
    \pgfkeysalso{blue}%
    \global\let\alternatecolor\alternatecolorred 
}

\newcommand{\altred}{\let\alternatecolor\alternatecolorred 
\tikzset{every edge/.append code = {%
    \global\let\currenttarget\tikztotarget 
    \pgfkeysalso{append after command={(\currenttarget)}}
			\alternatecolor
}}
}
\newcommand{\altblue}{\let\alternatecolor\alternatecolorblue 
\tikzset{every edge/.append code = {%
    \global\let\currenttarget\tikztotarget 
    \pgfkeysalso{append after command={(\currenttarget)}}
			\alternatecolor
}}
}

\tikzstyle{vertexdot}=[circle, draw, fill=black, minimum size=3,inner sep=0pt]

\usepackage{color}
\usepackage{subfigure}
\usepackage{algorithm}
\usepackage{algorithmic}

\usepackage{xspace,prettyref}

\newrefformat{eq}{(\ref{#1})}
\newrefformat{chap}{Chapter~\ref{#1}}
\newrefformat{sec}{Section~\ref{#1}}
\newrefformat{alg}{Algorithm~\ref{#1}}
\newrefformat{fig}{Fig.~\ref{#1}}
\newrefformat{tab}{Table~\ref{#1}}
\newrefformat{rmk}{Remark~\ref{#1}}
\newrefformat{clm}{Claim~\ref{#1}}
\newrefformat{def}{Definition~\ref{#1}}
\newrefformat{cor}{Corollary~\ref{#1}}
\newrefformat{lmm}{Lemma~\ref{#1}}
\newrefformat{prop}{Proposition~\ref{#1}}
\newrefformat{app}{Appendix~\ref{#1}}
\newrefformat{hyp}{Hypothesis~\ref{#1}}
\newrefformat{thm}{Theorem~\ref{#1}}
\newrefformat{ass}{Assumption~\ref{#1}}
\newrefformat{conj}{Conjecture~\ref{#1}}

\DeclareMathAlphabet{\varmathbb}{U}{bbold}{m}{n}

\renewcommand{\tilde}{\widetilde}

\usepackage{bbm}

\newcommand{\1}{\mathbbm{1}}

\pgfdeclarelayer{background}
\pgfdeclarelayer{foreground}
\pgfsetlayers{background,main,foreground}

\usepackage{mathtools}

\theoremstyle{plain}

\newtheorem*{theorem*}{Theorem}
\newtheorem*{proposition*}{Proposition}

\newcommand{\cN}{\mathcal{N}}

\newcommand{\RR}{\mathbb{R}}

\newcommand*{\kl}[2]{\operatorname{D}(#1\,\|\,#2)}

\newcommand*{\chis}[2]{\chi^2(#1, #2)}

\newcommand*{\E}{\mathbb E}

\newcommand*{\ep}{\varepsilon}
\newcommand*{\defeq}{\triangleq}
\newcommand*{\dd}{\, \mathrm{d}}

\newcommand*{\mmse}[2]{\operatorname{MMSE}_{#1}(#2)}
\newtheorem{theorem}{Theorem}
\newtheorem{proposition}{Proposition}

\newtheorem{lemma}{Lemma}
\theoremstyle{definition}
\newtheorem{assumption}{Assumption}
\newtheorem{definition}{Definition}

\DeclarePairedDelimiter\ceil{\lceil}{\rceil}

\title{The All-or-Nothing Phenomenon in Sparse Tensor PCA}

\author{
Jonathan Niles-Weed\thanks{Courant Institute of Mathematical Sciences and Center for Data Science, New York University; e-mail: {\tt jnw@cims.nyu.edu}. JNW acknowledges the support of the Institute for Advanced Study, where a portion of this research was conducted.}
\and
{ Ilias Zadik}\thanks{Center for Data Science, New York University ; e-mail: {\tt zadik@nyu.edu.} IZ is supported by a CDS Moore-Sloan postdoctoral fellowship.}
}
\begin{document}

\maketitle

\begin{abstract}

We study the statistical problem of estimating a rank-one sparse tensor corrupted by additive Gaussian noise, a model also known as sparse tensor PCA. We show that for Bernoulli and Bernoulli-Rademacher distributed signals and \emph{for all} sparsity levels which are sublinear in the dimension of the signal, the sparse tensor PCA model exhibits a phase transition called the \emph{all-or-nothing phenomenon}. This is the property that for some signal-to-noise ratio (SNR) $\mathrm{SNR_c}$ and any fixed $\epsilon>0$, if the SNR of the model is below $\left(1-\epsilon\right)\mathrm{SNR_c}$, then it is impossible to achieve any arbitrarily small constant correlation with the hidden signal, while if the SNR is above $\left(1+\epsilon \right)\mathrm{SNR_c}$, then it is possible to achieve almost perfect correlation with the hidden signal. The all-or-nothing phenomenon was initially established in the context of sparse linear regression, and over the last year also in the context of sparse 2-tensor (matrix) PCA, Bernoulli group testing, and generalized linear models. Our results follow from a more general result showing that for any Gaussian additive model with a discrete uniform prior, the all-or-nothing phenomenon follows as a direct outcome of an appropriately defined ``near-orthogonality" property of the support of the prior distribution. 
\end{abstract}

\section{Introduction}
A central question in information theory and statistics is to establish the fundamental limits for recovering a planted signal in high-dimensional models.
A common theme in these works is the presence of \emph{phase transitions}, where the behavior of optimal estimators changes dramatically at critical values.
An infamous example is the PCA transition, where one observes a matrix $\mathbf Y \in \RR^{p \times p}$ given by
\begin{equation*}
\mathbf Y = \sqrt{\beta p} \, \mathbf x \mathbf x^\top + \mathbf W\,,
\end{equation*}
where $\mathbf x$ is drawn uniformly from the unit sphere in $\RR^p$ and $\mathbf W$ is an independent Gaussian Wigner matrix.
When $\beta < 1$, then as $p \to \infty$ the leading eigenvector of $\mathbf Y$ is asymptotically uncorrelated with $\mathbf x$; on the other hand, when $\beta > 1$, the correlation between the hidden signal $\mathbf x$ and the leading eigenvector of $\mathbf Y$ remains positive as $p \to \infty$ \cite{BBP,FP-wigner,wigner-eigenvec}.

A number of recent works establish that several sparse estimation tasks in high dimensions evince the following even more striking phase transition, called the ``all-or-nothing'' phenomenon \cite{gamarnikzadik, ReevesPhenom, Zad19}: below a critical signal-to-noise ratio (SNR), it is impossible to achieve \emph{any} correlation with the hidden signal, but above this critical SNR, it is possible to achieve \emph{almost perfect} correlation with the hidden signal.
In other words, at any fixed SNR, one can either recover the signal perfectly, or nothing at all.

In prior work, the all-or-nothing phenomenon has been established in a smattering of different models and sparsity regimes.  Understanding the extent to which this phenomenon holds more generally, beyond the models and sparsity conditions previously considered, is the main motivation of the present work. The all-or-nothing phenomenon for sparse linear regression was initially conjectured by~\cite{gamarnikzadik}. That work established that a version of this phenomenon holds for the maximum likelihood estimator (MLE): below a given threshold, the MLE does not achieve any constant correlation with the hidden signal, while above the same threshold the MLE achieves almost perfect correlation. However, this result does not rule out the existence of other estimators with better performance. Subsequently,~\cite{ReevesPhenom} proved that the all-or-nothing phenomenon indeed holds for sparse linear regression, when the sparsity level $k$ satisfies $k \leq p^{\frac{1}{2}-\epsilon}$ for some $\epsilon>0$, where $p$ is the dimension of the model. Furthermore,~\cite{ReevesCAMSAP} provided generic conditions under which the phenomenon holds for sparse linear regression when $k/p=\delta>0$ where $\delta>0$ is a constant which shrinks to zero. Following these works, \cite{barbier01, barbier2020allornothing} showed that the sparse PCA model with binary and Rademacher non-zero entries also exhibits the all-or-nothing phenomenon when the sparsity level $k$ satisfies $p^{\frac{12}{13}+\epsilon} \leq k \ll p$. The ``all-or-nothing" phenomenon has also recently been established in the context of the Bernoulli group testing model \cite{ScarlletAll}. That work proves the existence of this phenomenon in an extremely sparse setting, where $k$ scales polylogarithmically with the dimension of the model. Finally, it was recently shown in \cite{luneau20} using analytical non-rigorous methods, that the all-or-nothing phenomenon also holds for various generalized linear models with a $k$-sparse signal, under the assumption $p^{\frac{8}{9}+\epsilon} \leq k \ll p.$

\subsection{Contribution}
In this work, in an attempt to shed some light on the fundamental reason for the existence of the all-or-nothing phenomenon, we focus on a simple Gaussian model, which we refer to as the Gaussian additive model, in which one observes a hidden signal drawn from some known prior, corrupted with additive Gaussian noise. For example, all PCA models, and in particular the sparse PCA model considered by \cite{barbier01, barbier2020allornothing}, are special cases of Gaussian additive models. We focus on the case where the prior is an arbitrary uniform distribution over some discrete subset of the Euclidean sphere.

We make the following contributions.
\begin{itemize}
\item We show that for this additive Gaussian model, the all-or-nothing phenomenon is equivalent to a simple criterion on the Kullback-Lieber divergence between the model and a null distribution with i.i.d.~Gaussian entries.
\item We show that, under an appropriate ``near-orthogonality'' condition on the prior, the all-or-nothing phenomenon always holds.
\item As an application, we study \emph{sparse tensor PCA}, in which the hidden signal is a rank-one tensor $\mathbf x^{\otimes d} \in (\RR^{p})^{\otimes d}$, where the entries of $x$ are $k$-sparse. We show that for both the Bernoulli and Bernoulli-Rademacher prior, all sparsity levels $k = o(p)$, and all $d \geq 2$, this model satisfies the aforementioned near-orthogonality condition, and therefore evinces the all-or-nothing phenomenon. This confirms a conjecture implicit in several prior works~\cite{barbier01,BanksIT,LesKrzZde17,PerWeiBan20, barbier2020allornothing}. To the best of our knowledge this is the first result that proves the all-or-nothing phenomenon \emph{for all sparsity levels which are sublinear in the dimension of the model}.
\end{itemize}

Omitted proofs and lemmas appear in the appendix.
\subsection{Comparison with previous work}
Our results for sparse tensor PCA are closely connected to several prior works.

\paragraph{\cite{BanksIT} and \cite{PerWeiBan20}}
These papers study the sparse tensor PCA problem with a Bernoulli-Rademacher prior. Their focus is on optimal recovery of the hidden signal in the regime where the sparsity satisfies $k=\gamma p$ for some constant $\gamma>0$. \cite{BanksIT} and \cite{PerWeiBan20} identify two thresholds, $\mathrm{SNR_{\text{lower}}}$ and $\mathrm{SNR_{\text{upper}}}$, such that below the first threshold, no constant correlation with the hidden signal is possible, while above the second threshold it is possible to obtain constant correlation with the signal. Interestingly, as $\gamma \rightarrow 0$, the two thresholds become identical. Both papers use a trick known as the conditional second moment method, and our argument in Section~\ref{sec:smm} is closely inspired by their techniques.

Our results differ from theirs in two important respects. First, though taking the sparse limit $\gamma \to 0$ is suggestive, these works do not offer a rigorous way to establish the presence of a threshold when $k = o(p)$.
More importantly, the results of \cite{BanksIT} and \cite{PerWeiBan20} elucidate the threshold between ``no recovery'' (zero correlation with hidden signal) and ``partial recovery'' (constant correlation with hidden signal) for sparse tensor PCA. By contrast, we focus on the much sharper transition between no recovery and almost perfect recovery.

\paragraph{\cite{barbier01} and \cite{barbier2020allornothing}}
Unlike \cite{BanksIT} and \cite{PerWeiBan20}, \cite{barbier01} and \cite{barbier2020allornothing} study the genuinely sublinear setting when $k = o(p)$.
While they prove very precise results characterizing the limiting free energy of the sparse (matrix) PCA problem, their techniques require that $k \geq p^{\frac{12}{13} + \ep}$ for some $\epsilon>0$. Our results are less fine, insofar as we do not precisely characterize the free energy for arbitrary sparsity and SNR, but we show that the all-or-nothing phenomenon holds for a much broader range of parameters via a much simpler argument.

\section{Main Results}
\subsection{General framework: the Gaussian Additive Model}
We consider throughout the following observation model which we refer to as a \emph{Gaussian additive model}:
\begin{equation}\label{observation_model}
\mathbf Y = \sqrt{\lambda} \mathbf X + \mathbf Z\,,
\end{equation}
where $\mathbf X \in \RR^{N}$ is drawn from a uniform discrete prior distribution $\mathrm P_N$ on the unit sphere in $\RR^N$ and $\mathbf Z \in \RR^N$ has i.i.d.~standard Gaussian entries.
We denote by $\mathrm Q_{\lambda, N}$ the law of $\mathbf Y$, where we use the subscripts $\lambda$ and $N$ to emphasize that this law depends on the signal-to-noise ratio $\lambda$ and the dimension $N$.
 
Given $\lambda \geq 0$, we let
\begin{equation*}
\mmse{N}{\lambda} \defeq \E \|\mathbf X - \E[\mathbf X | \mathbf Y]\|^2 \quad \quad \mathbf Y \sim \mathrm Q_{\lambda, N}\,,
\end{equation*}
where $\mathbf X$ and $\mathbf Y$ are as in~\eqref{observation_model}.
This quantity is the smallest mean squared error achievable by any estimator of $\mathbf X$ based on the observation $\mathbf Y$. The optimal estimator $\E[\mathbf X | \mathbf Y]$ is commonly referred to as the Bayes-optimal estimator.
The fact that $\|\mathbf X\| = 1$ almost surely implies that $\mmse{N}{\lambda} \leq 1$, since this mean-squared error is always achievable by a trivial estimator which is identically zero.

We say that a sequence of distributions $\{\mathrm P_N\}$ satisfies the the \emph{all-or-nothing phenomenon} with critical SNR $\{\lambda_N\}$ if
\begin{equation*}
\lim_{N \to \infty} \mmse{N}{\beta \lambda_N} = \left\{
\begin{array}{ll}
1 & \text{ if $\beta < 1$} \\
0 & \text{ if $\beta > 1$}\,.
\end{array}
\right.
\end{equation*}
In other words, above some critical value, it is possible to estimate $\mathbf X$ nearly perfectly, but below this critical value it is not possible to estimate $\mathbf X$ at all, in the sense that the best estimator is no better than the trivial zero estimator.

Recall that we have assumed that $\mathrm P_N$ is the uniform distribution on some finite subset.
Denote the cardinality of this subset by $M_N$.
We assume throughout that $M_N \to \infty$ as $N \to \infty$.
We also make the following assumption, which requires that the distribution $\mathrm P_N$ is sufficiently spread out.

\begin{assumption}\label{spread}
For independent draws $\mathbf X, \mathbf X'$ from $\mathrm P_N$, we have
\begin{equation*}
\lim_{t \to 1} \limsup_{N \to \infty} \frac{1}{\log M_N} \log \mathrm P_N^{\otimes 2}[\langle \mathbf X, \mathbf X' \rangle \geq t] \leq -1\,.
\end{equation*}
\end{assumption}
In other words, Assumption~\ref{spread} holds as long as the asymptotic probability that $\mathbf X$ and $\mathbf X'$ are very near each other is not much larger than the probability that $\mathbf X = \mathbf X'$.

\subsection{Main Results: the all-or-nothing behavior for the Gaussian Additive Model}
Our first main result shows that, under this assumption, there is an easy characterization of the priors which satisfy the all-or-nothing phenomenon.
We denote by $\mathrm D$ the Kullback-Leibler divergence (see, e.g., \cite[Section 6]{PW-it}); given two probability distributions $\mathrm P_1,\mathrm P_2$ with $\mathrm P_1$ absolutely continuous to $\mathrm P_2$,
$$ 
\kl{\mathrm P_1}{\mathrm P_2} \triangleq \Exp_{\mathrm P_2} \left[  \frac{ \mathrm{d P_1}}{ \mathrm{d P_2}}\log \left(  \frac{\mathrm{dP_1}}{ \mathrm d \mathrm P_2} \right)  \right].
$$

\begin{theorem}\label{area_thm}
Under Assumption~\ref{spread}, a sequence $\{\mathrm P_N\}$ satisfies the all-or-nothing phenomenon with critical SNR $\lambda_N$  if and only if $\kl{\mathrm Q_{2 \log M_N, N}}{\mathrm Q_{0, N}} = o(\log M_N)$. Moreover, in this situation, we can take $\lambda_N = 2 \log M_N$.
\end{theorem} 
To prove Theorem~\ref{area_thm}, we employ a well known connection between the Kullback-Liebler divergence and the MMSE, known as the I-MMSE relation (see, e.g., \cite{GuoShaVer05}), 
\begin{equation*}
\frac{d}{d \beta} \frac{1}{\lambda_N} \kl{\mathrm Q_{\beta \lambda_N, N}}{\mathrm Q_{0, N}}=\frac{1}{2}-\frac{1}{2} \mmse{N}{\beta \lambda_N}\,.
\end{equation*} 
This relation implies the following characterization.
\begin{proposition}\label{prop:kl_limit}
The all-or-nothing phenomenon holds with critical SNR $\lambda_N$ if and only if
\begin{equation*}
\lim_{N \to \infty} \frac{1}{\lambda_N} \kl{\mathrm Q_{\beta \lambda_N, N}}{\mathrm Q_{0, N}} = \frac 12 (\beta - 1)_+ \quad \quad \forall \beta \geq 0\,.
\end{equation*}
\end{proposition}
Then, the proof of Theorem~\ref{area_thm} exploits the fact that $\beta \mapsto \frac{1}{\lambda_N} \kl{\mathrm Q_{\beta \lambda_N, N}}{\mathrm Q_{0, N}}$ is nonnegative, increasing, and convex. Therefore, specifying the limit for a few well-chosen values of $\beta$ is enough to establish the entire limit.
We present a complete proof in Section \ref{sec:area_thm}.

One can naturally ask whether the above characterization yields any simple criteria for a prior to evince the all-or-nothing phenomenon.
Our next result shows that the all-or-nothing phenomenon is implied by a simple condition on the \emph{overlap} of two independent draws from $\mathrm P_N$.

The condition is outlined in the following definition.
\begin{definition}
Given a non-decreasing function $r: [-1, 1] \to \RR_{\geq 0}$, we say $\{\mathrm P_N\}$ has \emph{overlap rate function} $r$ if 
\begin{equation*}
\limsup_{N \to \infty} \frac{1}{\log M_N} \log \mathrm P_N^{\otimes 2}[\langle \mathbf X, \mathbf X' \rangle \geq t] \leq - r(t)\,,
\end{equation*}
where $\mathbf X$ and $\mathbf X'$ are independent draws from $\mathrm P_N$.
\end{definition}

Our following result shows that a lower bound on the growth of the overlap rate function suffices to establish the all-or-nothing phenomenon. 
\begin{theorem}\label{main_thm}
Suppose that $\{\mathrm{P}_N\}$ has overlap rate function $r$ satisfying for all $t \in [0,1]$,
 \begin{equation}\label{condition}
r(t) \geq \frac{2t}{1+t}.
\end{equation}
Then $\{\mathrm P_N\}$ satisfies the all-or-nothing phenomenon at $\lambda_N = 2 \log M_N$.
\end{theorem}
In words, the condition requests a particular decay condition on the upper tail of the overlap. For example, notice that the condition is trivially satisfied when the support consists of pairwise orthogonal vectors, since in that case $r(t)=1$ for all $t \in (0,1]$. Hence, the all-or-nothing phenomenon holds for any uniform prior distribution supported on a family of orthogonal vectors on the sphere. In the next section we present more complicated examples of prior distributions satisfying this condition. The proof of Theorem \ref{main_thm} is presented in Section \ref{sec:smm}.

\subsection{Application: the sparse tensor PCA model}\label{sec:sparseTensor}
We apply our framework to a well-studied inference model called sparse tensor PCA, and show that it exhibits the all-or-nothing phenomenon for all sublinear sparsity levels.

For some $d \geq 2$, we define first the  \emph{tensor PCA model} to be the model given in \eqref{observation_model}, where the vectors $\mathbf Y$ and $\mathbf Z$ live in dimension $N = p^d$ and the discrete prior distribution $\mathrm{P}_N$ is supported on a subset of the vectorized $d$-tensors $\mathbf X = \mathbf x^{\otimes d}$, where this notation refers to the vector whose entry indexed by $(i_1, \dots, i_d) \in [p]^d$ is $\mathbf x_{i_1} \cdots \mathbf{x}_{i_d}$. We assume the vector $\mathbf x \in \RR^p$ is drawn from a discrete distribution $\tilde{ \mathrm  P}_p$  on the unit sphere in $\RR^p$, which induces a natural prior distribution $\mathrm{P}_N$ on the tensors $\mathbf x^{\otimes d}$.

We define the \emph{sparse tensor PCA model} to be the above tensor PCA model with one of the following two prior distributions:
\begin{description}
\item[Bernoulli:] $\tilde{\mathrm P}_p$ is the uniform distribution over the subset of $\{0,1/\sqrt{k} \}^p$ with exactly $k$ nonzero entries.
\item[Bernoulli-Rademacher:] $\tilde{\mathrm P}_p$ is the uniform distribution over the subset of $\{-1/\sqrt{k} ,0,1/\sqrt{k} \}^p$ with exactly $k$ nonzero entries.
\end{description}

In the appendix, we prove the following elementary bound.
\begin{proposition}\label{prop:sparse_tensor_overlap}
Suppose that $k=o(p)$ and that $\tilde{\mathrm P}_p$ is either Benoulli or Bernoulli-Rademacher, and let ${\mathrm P}_N$ be the induced prior distribution on $\RR^N = \RR^{p^d}$.
Then for any $t \in [0,1]$ it holds
\begin{equation*}
\lim_{N \rightarrow +\infty} \frac{1}{\log M_N} \log \mathrm P_N^{\otimes 2}[\langle \mathbf X, \mathbf X' \rangle \geq t] \leq - \frac{2t}{1+t}\,.
\end{equation*}
\end{proposition}
Combining this bound with Theorem~\ref{main_thm} immediately yields our main result for sparse tensor PCA.

\begin{theorem}\label{thm:sparse}
For any $d \geq 2$ and $k = o(p)$, the sparse tensor PCA model
\begin{equation*}
\mathbf Y =\beta \sqrt{2k \log \left(\frac{p}{k}\right)} \mathbf x^{\otimes d} + \mathbf Z\,, \quad \mathbf x \sim \tilde{\mathrm P}_p
\end{equation*}
with Bernoulli or Bernoulli-Rademacher prior exhibits the all-or-nothing phenomenon:
\begin{equation*}
\lim_{N \to \infty} \E \| \mathbf x^{\otimes d} - \E[ \mathbf x^{\otimes d}  | \mathbf Y]\|^2  = \left\{
\begin{array}{ll}
1 & \text{ if $\beta < 1$} \\
0 & \text{ if $\beta > 1$}\,.
\end{array}
\right.
\end{equation*}
\end{theorem}

\section{Proof of Theorem \ref{area_thm}}\label{sec:area_thm}
In this section, we present the proof of our main equivalence, Theorem~\ref{area_thm}.
As noted above, it is a consequence of Proposition \ref{prop:kl_limit}, whose proof appears in the appendix.

We first show that if $\kl{\mathrm Q_{2 \log M_N, N}}{\mathrm Q_{0, N}} = o(\log M_N)$, then $\{\mathrm P_N\}$ satisfies the all-or-nothing phenomenon with critical SNR equal to $2 \log M_N$.
Setting $\lambda_N = 2 \log M_N$, we have by assumption that
\begin{equation}\label{beta=1_limit}
\lim_{N \to \infty} \frac{1}{\lambda_N} \kl{\mathrm Q_{\lambda_N, N}}{\mathrm Q_{0, N}} = 0\,,
\end{equation}
which, since $\kl{\mathrm Q_{\beta \lambda_N, N}}{\mathrm Q_{0, N}}$ is nonnegative and nondecreasing as a function of $\beta$ (Lemma~\ref{lem:kl_facts}), implies that
\begin{equation}\label{beta<1_limit}
\lim_{N \to \infty} \frac{1}{\lambda_N} \kl{\mathrm Q_{\beta \lambda_N, N}}{\mathrm Q_{0, N}} = 0 \quad \quad \forall \beta \in [0, 1]\,.
\end{equation}

By Lemma~\ref{lem:lb}, we see
\begin{equation*}
\liminf_{N \to \infty} \frac{1}{\lambda_N} \kl{\mathrm Q_{\beta \lambda_N, N}}{\mathrm Q_{0, N}} \geq \frac 12 (\beta - 1)\,.
\end{equation*}
However, since $\frac{1}{\lambda_N} \kl{\mathrm Q_{\beta \lambda_N, N}}{\mathrm Q_{0, N}}$ is $\frac 12$-Lipschitz (Lemma~\ref{lem:kl_facts}), we have
\begin{equation*}
\limsup_{N \to \infty} \frac{1}{\lambda_N} \kl{\mathrm Q_{\beta \lambda_N, N}}{\mathrm Q_{0, N}} \leq \frac 12 |\beta - 1| + \limsup_{N \to \infty} \frac{1}{\lambda_N} \kl{\mathrm Q_{\lambda_N, N}}{\mathrm Q_{0, N}} = \frac 12 |\beta - 1|\,.
\end{equation*}
We therefore obtain, for $\beta \geq 1$,
\begin{equation*}
\lim_{N \to \infty} \frac{1}{\lambda_N} \kl{\mathrm Q_{\beta \lambda_N, N}}{\mathrm Q_{0, N}} = \frac 12 (\beta - 1)\,.
\end{equation*}
Combined with~\eqref{beta<1_limit}, we obtain via Proposition~\ref{prop:kl_limit} that $\{\mathrm P_N\}$ satisfies the all-or-nothing phenomenon with critical SNR $2 \log M_N$.

In the other direction, we suppose that the all-or-nothing phenomenon holds with some SNR $\lambda_N$.
By Lemma~\ref{lem:mutual_information}, we can write
\begin{equation}\label{eq:mutual_information}
\frac{\lambda_N \beta}{2} - \kl{\mathrm Q_{\beta \lambda_N, N}}{\mathrm Q_{0, N}} = \kl{\mathrm Q_{\beta \lambda_N, N}^{(\mathbf X, \mathbf Y)}}{\mathrm P_N \otimes \mathrm Q_{\beta \lambda_N, N} }\,,
\end{equation}
where $\mathrm Q_{\beta \lambda_N, N}^{(\mathbf X, \mathbf Y)}$ indicates the joint law of $\mathbf X, \mathbf Y$ generated according to~\eqref{observation_model}.

Given an observation $\mathbf Y$, let us denote by $\mathbf X'$  a sample from the conditional distribution $\mathrm{P}_N \mid \mathbf Y$.
If $(\mathbf X, \mathbf Y) \sim \mathrm Q_{\beta \lambda_N, N}^{(\mathbf X, \mathbf Y)}$, then this induces a joint distribution on $(\mathbf X, \mathbf X')$ which we denote by $\mathrm P_{\beta \lambda_N, N}$.
On the other hand, if $\mathbf X$ and $\mathbf Y$ are independent, then $\mathbf X$ and $\mathbf X'$ are independent and marginally each has distribution $\mathrm P_N$, so the pair $(\mathbf X, \mathbf X')$ has law $\mathrm P^{\otimes 2}_N$.

Applying the data processing inequality twice, we obtain for any event $\Omega$ that
\begin{equation*}
\kl{\mathrm Q_{\beta \lambda_N, N}^{(\mathbf X, \mathbf Y)}}{ \mathrm P_N \otimes \mathrm Q_{\beta \lambda_N, N}} \geq \kl{\mathrm P_{\beta \lambda_N, N}}{\mathrm P^{\otimes 2}_N}\ \geq  d(\mathrm P_{\beta \lambda_N, N}(\Omega) \,\|\, \mathrm P^{\otimes 2}_N(\Omega))\,,
\end{equation*}
where $d$ is the binary divergence function:
\begin{equation*}
d(\alpha_1 \,\|\, \alpha_2) \defeq \alpha_1 \log \frac{\alpha_1}{\alpha_2} + (1- \alpha_1) \log \frac{1-\alpha_1}{1-\alpha_2}\,.
\end{equation*}

Fix a $t \in [0, 1)$, and set
\begin{equation*}
\Omega_t \defeq \{(x, x'): \langle x, x' \rangle \geq t\}\,.
\end{equation*}
Suppose $(\mathbf X, \mathbf X') \sim \mathrm P_{\beta \lambda_N, N}$.
For any $\beta > 1$, the fact that the all-or-nothing phenomenon holds implies that
\begin{equation*}
\E \langle \mathbf X, \mathbf X' \rangle = \E \langle \mathbf X, \E[\mathbf X | \mathbf Y] \rangle = 1 - o(1)\,,
\end{equation*}
Since $\langle \mathbf X, \mathbf X' \rangle \leq 1$ almost surely, this implies that we must also have
\begin{equation*}
\lim_{N \to \infty} \mathrm P_{\beta \lambda_N, N}(\Omega_t) = 1.
\end{equation*}

On the other hand, Assumption~\ref{spread} implies that for $t$ sufficiently close to $1$,
\begin{equation*}
\lim_{N \to \infty} \mathrm P^{\otimes 2}_{N}(\Omega_t) = 0\,.
\end{equation*}

Combining these observations, we obtain for any $\beta > 1$ and $t$ sufficiently close to $1$,
\begin{align*}
\limsup_{N \to \infty} \frac{1}{\lambda_N} \kl{\mathrm Q_{\beta \lambda_N, N}(\mathbf X, \mathbf Y)}{\mathrm P_N \otimes \mathrm Q_{\beta \lambda_N, N} } & \geq
\limsup_{N \to \infty} \frac{1}{\lambda_N} d(\mathrm P_{\beta \lambda_N, N}(\Omega_t) \,\|\, \mathrm P^{\otimes 2}_N(\Omega_t)) \\
& = \limsup_{N \to \infty}  \frac{1}{\lambda_N} \log \frac{1}{\mathrm P^{\otimes 2}_{N}(\Omega_t)}\,,
\end{align*}
where we justify the final limit in Lemma~\ref{binary_divergence}, noting that~\eqref{eq:mutual_information} implies that $\frac{1}{\lambda_N} \kl{\mathrm Q_{\beta \lambda_N, N}(\mathbf X, \mathbf Y)}{\mathrm P_N \otimes \mathrm Q_{\beta \lambda_N, N} }$ is bounded.

Under the all-or-nothing phenomenon, Proposition~\ref{prop:kl_limit} and~\eqref{eq:mutual_information} imply that the left side of the above inequality is $1/2$.
Combining this with Assumption~\ref{spread}, we obtain that for any $\delta > 0$, there exists $t \in [0, 1)$ such that for all $N$ sufficiently large,
\begin{align*}
\frac{1}{\lambda_N} \log  \mathrm P^{\otimes 2}_{N}(\Omega_t) & \geq - \frac 12 - \delta \\
\frac{1}{\log M_N} \log  \mathrm P^{\otimes 2}_{N}(\Omega_t) & \leq -1 + \delta
\end{align*}
In particular, we must have $\lambda_N \geq (2 - O(\delta)) \log M_N$ for all $N$ large enough, so that
\begin{equation*}
\liminf_{N \to \infty} \frac{\lambda_N}{\log M_N} \geq 2 - O(\delta)\,,
\end{equation*}
and letting $\delta \to 0$ yields
\begin{equation*}
\liminf_{N \to \infty} \frac{\lambda_N}{\log M_N} \geq 2\,.
\end{equation*}
On the other hand, Lemma~\ref{lem:lb} implies
\begin{equation*}
\lim_{N \to \infty} \frac{1}{\lambda_N} \kl{\mathrm Q_{\beta \lambda_N, N}}{\mathrm Q_{0, N}} \geq \limsup_{N \to \infty} \left\{\frac 12 \beta - \frac{\log M_N}{\lambda_N}\right\}\,,
\end{equation*}
which, combined with Proposition~\ref{prop:kl_limit} for some fixed $\beta>1$, yields
\begin{equation*}
\limsup_{N \to \infty} \frac{\lambda_N}{\log M_N} \leq 2\,.
\end{equation*}

Hence, for $\ep > 0$,
\begin{align*}
\limsup_{N \to \infty} \frac{1}{\log M_N} \kl{\mathrm Q_{2 \log M_N, N}}{\mathrm Q_{0, N}} & \leq \limsup_{N \to \infty} \frac{1}{\log M_N} \kl{\mathrm Q_{(1+\ep) \lambda_N, N}}{\mathrm Q_{0, N}} \\
& = 2 \limsup_{N \to \infty} \frac{1}{\lambda_N} \kl{\mathrm Q_{(1+\ep) \lambda_N, N}}{\mathrm Q_{0, N}} \\
& = \ep\,.
\end{align*}
Taking $\ep \to 0$ yields that $\kl{\mathrm Q_{2 \log M_N, N}}{\mathrm Q_{0, N}} = o(\log M_N)$ and shows that we can take $\lambda_N = 2 \log M_N$.
\qed

\section{Proof of Theorem \ref{main_thm}: A conditional second moment method}\label{sec:smm}
In this section, we employ an argument known as the ``conditional second moment method'' to show Theorem \ref{main_thm}. We make use of the following definition which is essentially borrowed from~\cite[Section 3.3]{BanksIT}.
\begin{definition}\label{dfn:highprobevents}
Write $\mathrm Q_{\lambda, N}^{(\mathbf X, \mathbf Y)}$ for the joint distribution of $(\mathbf X, \mathbf Y)$ in~\eqref{observation_model}.
Given a sequence $\lambda_N$, we say that a sequence of events $\Omega_N, N \in \mathbb{N}$  occurs with \emph{uniformly high probability} if as $N \rightarrow + \infty$, \begin{equation*}
\mathrm Q_{\lambda_N, N}^{(\mathbf X, \mathbf Y)}\left[ \Omega_N |\mathbf X=x  \right]=1-o(1), \end{equation*} uniformly over all $x$ in the support of $\mathrm P_N$.
\end{definition}

Given such a sequence, we write $\tilde{\mathrm Q}_{\lambda_N, N}$ for the marginal law of $\mathbf Y$ when we condition on the event~$\Omega_N$.
In the appendix, we establish the following proposition.

\begin{proposition}\label{prop:condKL}
Let $\lambda_N = 2 \log M_N$.
If $\Omega_N$ is a sequence of uniform high probability events, then as $N \rightarrow + \infty$, it holds
\begin{equation*}
 \kl{\mathrm Q_{\lambda_N, N}}{\mathrm Q_{0, N}} \leq   \kl{\tilde{\mathrm Q}_{\lambda_N, N}}{\mathrm Q_{0, N}}+o\left( \log M_N\right)\,.
\end{equation*}
\end{proposition}

We now establish the following.
\begin{theorem}\label{thm:conditional_second_moment}
Assume that $\{\mathrm P_N\}$ has overlap rate function $r$.
Let $\lambda_N = 2 \log M_N$.
There exists a uniformly high probability sequence of events $\Omega_N$ such that
\begin{equation*}
\limsup_{N \to \infty} \frac{1}{\lambda_N} \kl{\tilde{\mathrm Q}_{\lambda_N, N}}{\mathrm Q_{0, N}} \leq \sup_{t \in [0, 1]} \left(\frac{t}{1 + t} - \frac{r(t)}{2}\right)_+\,.
\end{equation*}
\end{theorem}

The hypothesis that $r(t) \geq \frac{2t}{1+t}$ implies that Assumption~\ref{spread} holds, so Theorem \ref{main_thm} follows immediately by combining Proposition \ref{prop:condKL}, Theorem \ref{area_thm} and Theorem~\ref{thm:conditional_second_moment}. For the rest section we focus on proving Theorem~\ref{thm:conditional_second_moment}.

\begin{proof}[Proof of Theorem~\ref{thm:conditional_second_moment}]
We define
\begin{equation*}
\Omega_N = \{(x, y): |\langle x, y \rangle - \sqrt{\lambda_N}| \leq \lambda_N^{1/4}\}\,.
\end{equation*}
Since $\lambda_N \to \infty$, the sequence $\Omega_N$ occurs with uniformly high probability.

We then have
\begin{equation*}
\frac{\mathrm d \mathrm {\tilde Q}_{\lambda_N, N}}{\mathrm d \mathrm Q_{0, N}}(\mathbf Y) = (1+o(1)) \E_{\mathbf X} \left\{\1_{\Omega_N}(\mathbf X, \mathbf Y) \cdot \exp\Big(\sqrt{\lambda_N} \langle \mathbf X, \mathbf Y \rangle - \frac{\lambda_N}{2}\Big)\right\}\,,
\end{equation*}
which implies
\begin{equation*}
\left(\frac{\mathrm d \mathrm{\tilde Q}_{\lambda_N, N}}{\mathrm d \mathrm Q_{0, N}}\right)^2(\mathbf Y) = (1+o(1)) \E_{\mathbf X, \mathbf X'} \left\{\1_{\Omega_N}(\mathbf X, \mathbf Y)\1_{\Omega_N}(\mathbf X', \mathbf Y) \cdot \exp\Big(\sqrt{\lambda_N} \langle \mathbf X + \mathbf X', \mathbf Y \rangle - \lambda_N\Big)\right\}\,,
\end{equation*}
where $\mathbf X, \mathbf X'$ are independent and identically distributed.
Applying Fubini's theorem, we obtain that the chi-square divergence satisfies
\begin{equation*}
1 + \chis{\tilde{\mathrm Q}_{\lambda_N, N}}{\mathrm Q_{0, N}} = (1+o(1)) \E_{\mathbf X, \mathbf X'}[m_N(\mathbf X, \mathbf X')]\,,
\end{equation*}
where
\begin{equation*}
m_N(X, X') \defeq \E \left\{\1_{\Omega_N}(X, \mathbf Z)\1_{\Omega_N}( X', \mathbf Z) \cdot \exp\Big(\sqrt{\lambda_N} \langle  X +  X', \mathbf Z \rangle - \lambda_N\Big)\right\}\,, \quad \mathbf Z \sim \mathrm Q_{0, \lambda_N}\,.
\end{equation*}
By the rotational invariance of the Gaussian measure, it is straightforward to show that $m_N$ depends only on the overlap $\langle X,  X' \rangle$ between $ X$ and $X'$, which we denote by $\rho$.
We require the following proposition.
\begin{proposition}\label{prop:mn_bound}
There exists a constant $C>0$ such that for all $\rho \in [-1, 1]$, we have
\begin{equation*}
\frac{1}{\lambda_N} \log m_N(\rho) \leq \left(\frac{\rho}{1+\rho}\right)_+ + \frac{C}{\lambda_N^{1/4}}\,,
\end{equation*}
where $(x)_+ \defeq \max\{x, 0\}$, and where we set $(\rho/(1+\rho))_+ = 0$ for $\rho = -1$.
\end{proposition}

We defer the proof of Proposition~\ref{prop:mn_bound} to the appendix and show how this claim implies the theorem.
Recall that $\kl{\tilde{\mathrm Q}_{\lambda_N, N}}{\mathrm Q_{0, N}} \leq \log(1 + \chis{\tilde{\mathrm Q}_{\lambda_N, N}}{\mathrm Q_{0, N}})$.
We therefore know that
\begin{equation}\label{KL_chi^2}
\limsup_{N \to \infty} \frac{1}{\lambda_N} \kl{\tilde{\mathrm Q}_{\lambda_N, N}}{\mathrm Q_{0, N}} \leq \limsup_{N \to \infty} \frac{1}{\lambda_N} \log \E [m_N(\rho)]\,, \quad \rho = \langle \mathbf X, \mathbf X' \rangle\,.
\end{equation}

We employ a standard large deviations argument.
Fix a positive integer $k$. We have
\begin{align*}
\E[m_N(\rho)] & \leq \sum_{\ell = -k}^{k-1} \mathrm P^{\otimes 2}_N[\rho \geq \ell/k] \sup_{t \in [\ell/k, (\ell+1)/k)} m_N(t) \\
& \leq 2k \cdot \max_{-k \leq \ell < k} \sup_{t \in [\ell/k, (\ell+1)/k)} \exp\left(\lambda_N \left(\frac{t}{1 + t}\right)_+ + \log \mathrm P^{\otimes 2}_N[\rho \geq \ell/k] + C\lambda_N^{3/4}\right) \\
& \leq 2k \cdot \sup_{t \in [-1, 1]} \exp\left(\lambda_N \left(\frac{t}{1 + t}\right)_+ + \log \mathrm P^{\otimes 2}_N[\rho \geq t] + C\lambda_N^{3/4} + O\left(\frac{\lambda_N}k\right)\right)\,.
\end{align*}
Since $\lambda_N \to \infty$, we have
\begin{equation*}
\limsup_{N \to \infty} \frac{1}{\lambda_N} \log \E[m_N(\rho)] \leq \sup_{t \in [-1, 1]} \left\{\left(\frac{t}{1 + t}\right)_+ -  \frac{r(t)}{2}\right\} + O(1/k)\,,
\end{equation*}
and letting $k \to \infty$ and using~\eqref{KL_chi^2} yields
\begin{equation*}
\limsup_{N \to \infty} \frac{1}{\lambda_N} \kl{\tilde{\mathrm Q}_{\lambda_N, N}}{\mathrm Q_{0, N}} \leq \sup_{t \in [-1, 1]} \left\{\left(\frac{t}{1 + t}\right)_+ -  \frac{r(t)}{2}\right\}\,.
\end{equation*}
Note that since $r$ is an overlap rate function, we must have $r(-1) = 0$; therefore, we obtain that the supremum on the right side is nonnegative.
However, since the quantity in question is nonpositive when $t < 0$, we can restrict to the interval $t \in [0, 1]$ without loss of generality.
This proves the claim.
\end{proof}

\section{Conclusion}
This work shows that the all-or-nothing phenomenon in Gaussian additive models is equivalent to a condition on the Kullback-Leibler divergence between the model at a particular SNR and a standard Gaussian vector.
Using this equivalence, we derive a simple condition on the overlaps which guarantees the existence of the all-or-nothing phenomenon, and as a corollary show that this phenomenon indeed arises in sparse tensor PCA for all sublinear sparsity levels.

While this paper gives a characterization of the all-or-nothing phenomenon for Gaussian models, we leave open the question of whether our framework can be extended to a more general setting.
Neither the results of~\cite{ReevesPhenom} for sparse linear regression, nor of~\cite{ScarlletAll} for Bernoulli group testing, nor of~\cite{luneau20} for sparse generalized linear models are immediately implied by our main theorems.
It may yet be possible to obtain more general results which encompass all of these settings.

Our results hold for the Bayes-optimal estimator $\E[\mathbf X | \mathbf Y]$, and we conjecture that the maximum likelihood estimator is also optimal in this setting and thereby also exhibits the all-or-nothing phenomenon. However, neither of these estimators is computationally efficient in general. Interestingly, in \cite{ReevesCAMSAP} the authors provide evidence that, in the context of sparse linear regression, the all-or-nothing phenomenon holds for the performance of a well-studied computationally efficient algorithm called Approximate Message Passing (AMP). Very recently, and more relevant to this work, \cite{barbier2020allornothing} established that the all-or-nothing phenomenon for the performance of AMP indeed holds in the context of sparse (matrix) PCA when the sparsity of the signal $k$ satisfies $k=\Theta\left(n/(\log n)^A \right)$ for some constant $A>1$. Developing a general theory for the optimal recovery thresholds for certain families of polynomial-time estimators---and establishing whether similar all-or-nothing behavior holds---is an important question for future work.

\bibliographystyle{alpha}
\bibliography{papers}
\newpage
\appendix
\section{Omitted proofs}
\begin{proof}[Proof of Proposition~\ref{prop:kl_limit}]
We recall the I-MMSE relation~\eqref{immse}:
\begin{equation*}
\frac{d}{d \beta} \frac{1}{\lambda_N} \kl{\mathrm Q_{\beta \lambda_N, N}}{\mathrm Q_{0, N}}=\frac{1}{2}-\frac{1}{2} \mmse{N}{\beta \lambda_N}\,.
\end{equation*}

Let us first assume that the all-or-nothing phenomenon holds.
Since $\kl{\mathrm Q_{0, N}}{\mathrm Q_{0, N}} = 0$, we can write
\begin{align*}
\lim_{N \to \infty} \frac{1}{\lambda_N} \kl{\mathrm Q_{\beta \lambda_N, N}}{\mathrm Q_{0, N}} & = \lim_{N \to \infty} \int_{0}^\beta \frac{d}{d \kappa} \frac{1}{\lambda_N} \kl{\mathrm Q_{\kappa \lambda_N, N}}{\mathrm Q_{0, N}} \dd \kappa \\
& = \lim_{N \to \infty} \int_{0}^\beta \frac 12 - \frac 12 \mmse{N}{\kappa \lambda_N} \dd \kappa\\
& \overset{(*)}{=} \int_{0}^\beta \frac 12 - \lim_{N \to \infty} \mmse{N}{\kappa \lambda_N} \dd \kappa \\
& = \frac 12 (\beta - 1)_+\,,
\end{align*}
where in $(*)$ we have used the dominated convergence theorem and the fact that $\mmse{N}{\kappa \lambda_N} \in [0, 1]$ and where the last equality follows from the all-or-nothing phenomenon.

In the other direction, we use the fact that $\mmse{N}{\beta \lambda_N}$ is a non-increasing function of $\beta$~\cite[see, e.g.,][Proposition 1.3.1]{Mio19}.
Combined with the I-MMSE relation, this immediately yields that $\frac{1}{\lambda_N} \kl{\mathrm Q_{\beta \lambda_N, N}}{\mathrm Q_{0, N}}$ is convex.
We therefore have by standard facts in convex analysis~\cite[Proposition 4.3.4]{HirLem93} that
\begin{equation*}
\frac 12 - \frac 12 \lim_{N \to \infty} \mmse{N}{\beta \lambda_N} = \lim_{N \to \infty} \frac{d}{d \beta} \frac{1}{\lambda_N} \kl{\mathrm Q_{\beta \lambda_N, N}}{\mathrm Q_{0, N}} = \frac{d}{d \beta} \left(\lim_{N \to \infty} \frac{1}{\lambda_N} \kl{\mathrm Q_{\beta \lambda_N, N}}{\mathrm Q_{0, N}}\right)
\end{equation*}
for all $\beta$ for which the right side exists.
Since we have assumed that $$\lim_{N \to \infty} \frac{1}{\lambda_N} \kl{\mathrm Q_{\beta \lambda_N, N}}{\mathrm Q_{0, N}}= \frac 12 (\beta - 1)_+,$$ the right side is $0$ when $\beta < 1$ and $\frac 12$ when $\beta > 1$.
The all-or-nothing property immediately follows.
\end{proof}
\begin{proof}[Proof of Proposition \ref{prop:sparse_tensor_overlap}]
Denote by $\mathcal S_k$ the set of $k$-sparse vectors in $\RR^p$.
Note that the cardinality of $\{0,1/\sqrt{k} \}^p \cap \mathcal S_k$ is $\binom{p}{k}$ and the cardinality of $\{-1/\sqrt{k} ,0,1/\sqrt{k} \}^p \cap \mathcal S_k$ is $\binom{p}{k} 2^k$.
In the case of the Bernoulli prior, the identification $\mathbf x \mapsto x^{\otimes d}$ is a bijection, so $M_N$ for the Bernoulli prior is $\binom{p}{k}$.
In the case of the Bernoulli-Rademacher prior, when $d$ is odd the map $\mathbf x \mapsto x^{\otimes d}$ is still a bijection, but when $d$ is even, the vectors $\mathbf x$ and $-\mathbf x$ give rise to the same tensor. Therefore $M_N$ for the Bernoulli-Rademacher prior is either $\binom{p}{k} 2^k$ or $\binom{p}{k} 2^{k-1}$.
Nevertheless, using Stirling's approximation, since $k=o(p)$, we have for both the Bernoulli and Bernoulli-Rademacher prior that
\begin{equation*}
\log M_N = (1+o(1))k \log \frac p k\,.
\end{equation*}

Now notice that the overlap $\langle \mathbf X, \mathbf X' \rangle $ in the case that $\mathbf x$ is Bernoulli-Rademacher is stochastically dominated by the overlap when $\mathbf x$ is Bernoulli. To prove this, let us consider the natural coupling between the two different priors on $\mathbf x$: we first sample $\mathbf x_1$ from the sparse Bernoulli distribution and then choose uniformly at random the signs for the non-zero values of $\mathbf x_1$ to form a sample $\mathbf x_2$ from the Bernoulli-Rademacher distribution. Notice that by triangle inequality under this coupling it holds almost surely $$\mathbf \langle \mathbf x_2^{\otimes d}, \mathbf x_2'^{\otimes d} \rangle  \leq |\langle \mathbf x_2^{\otimes d}, \mathbf x_2'^{\otimes d} \rangle   | \leq \langle \mathbf x_1^{\otimes d}, \mathbf x_1'^{\otimes d} \rangle .$$
For this reason it suffices to prove our result only in the case the prior $\tilde{\mathrm P}_p$ is the uniform distribution over $\{0,1/\sqrt{k} \}^p \cap \mathcal S_k$.
We therefore focus on this case in the rest of the proof.

Now fix any $t \in [0,1]$ and notice that by elementary algebra for any $v,v' \in \RR^p$ with $\|v\|=\|v'\|=1$ since $d \geq 2$ it holds $\langle v^{\otimes d}, v'^{\otimes d} \rangle =\langle v, v' \rangle^d \leq \langle v, v' \rangle^2$. Hence as $\mathbf x, \mathbf x'$ live on the sphere of dimension $p$,
\begin{align}
 \mathrm P_N^{\otimes 2} [\langle \mathbf X, \mathbf X' \rangle \geq t] = \tilde{\mathrm P}_p^{\otimes 2}[\langle \mathbf x^{\otimes d}, \mathbf x'^{\otimes d} \rangle \geq t] &=\tilde{\mathrm P}_p^{\otimes 2}[\langle \mathbf x, \mathbf x' \rangle^d \geq t] \nonumber\\
& \leq  \tilde{\mathrm P}_p^{\otimes 2}[\langle \mathbf x, \mathbf x' \rangle^2 \geq t] \nonumber\\
& =  \tilde{\mathrm P}_p^{\otimes 2}[\langle \mathbf x, \mathbf x' \rangle \geq \sqrt{t}]. \label{eq:algebra}
\end{align} Since $\mathbf x, \mathbf x'$ are drawn from the uniform distribution over $\{0,1/\sqrt{k} \}^p \cap \mathcal S_k$, Lemma \ref{lem:hypergeom} combined with \eqref{eq:algebra} yields
\begin{equation*}
\lim_{N \rightarrow +\infty} \frac{1}{\log M_N} \log \mathrm P_N^{\otimes 2}[\langle \mathbf X, \mathbf X' \rangle \geq t] \leq - \sqrt{t}.
\end{equation*} The elementary inequality $-\sqrt{t} \leq -\frac{2t}{1+t}$ concludes the proof.
\end{proof}

\begin{proof}[Proof of Proposition \ref{prop:condKL}]
Let
\begin{equation*}
Z\left(Y\right)=\frac{\mathrm Q_{\lambda_N,N}\left(Y\right)}{\mathrm Q_{0,N}\left(Y\right) }=\mathbb{E}_{\mathbf X' \sim \mathrm P_N} \exp \left( \sqrt{\lambda_N} \langle Y,\mathbf X'  \rangle -\frac{\lambda_N}{2}\right)
\end{equation*}

Following \emph{mutatis mutandis} the first two arguments in the proof of \cite[Theorem 5]{BanksIT} we obtain
\begin{equation}\label{eq:banks}
\kl{\mathrm Q_{\lambda_N,N}}{\mathrm  Q_{0,N}} \leq   \kl{\tilde{\mathrm Q}_{\lambda_N,N}}{\mathrm Q_{0,N}}+o\left( 1\right)\cdot \sqrt{ \E_{\mathbf Y \sim \mathrm Q_{\lambda_N, N}}\left[ \log^2 Z\left(\mathbf Y\right)\right]}.
\end{equation}

It is straightforward to see that for all $Y$,
\begin{equation*} |\log Z(Y)| \leq \sqrt{\lambda_N}\max_{X' \in \mathrm{Support}(P_N)} \langle X', Y \rangle+\frac{\lambda_N}{2}
\end{equation*} which implies that
\begin{equation}\label{eq:quadratic_ineq} \E_{\mathbf Y \sim \mathrm Q_{\lambda_N,N} } \log^2 Z(\mathbf Y) \leq 2\lambda_N  \cdot \E_{\mathbf Y \sim \mathrm Q_{\lambda_N,N}}\max_{X' \in \mathrm{Support}(P_N)} \langle X', \mathbf Y \rangle^2+O\left(\lambda^2_N\right).
\end{equation} Now recall $\mathbf Y=\sqrt{\lambda_N}\mathbf X+\mathbf Z$ for $\mathbf Z \sim Q_{0,N}$ and for all $X' \in \mathrm{Support}(P_N)$ it holds $|\langle \mathbf X , X'\rangle| \leq \|\mathbf X\|\|X'\| =1$ almost surely.
Hence,
\begin{align*}
 \E_{\mathbf Y \sim Q_{\lambda_N,N} }\max_{X' \in \mathrm{Support}(P_N)} \langle X', \mathbf Y \rangle^2 & =  \E_{\mathbf Z \sim Q_{0,N}  }\left(\max_{X' \in \mathrm{Support}(P_N)} |\sqrt{\lambda_N}\langle X', \mathbf X \rangle +   \langle X',  \mathbf Z \rangle| \right)^2 \\
& \leq 2 \lambda_N +2\E_{\mathbf Z \sim \mathrm Q_{0,N}} \max_{X' \in \mathrm{Support}(P_N)}  \langle X',  \mathbf Z \rangle^2.
\end{align*}
Since $\mathbf Q_{0,N}$ is simply the law of a vector with i.i.d. standard Gaussian coordinates and the cardinality of the discrete subset of the sphere $\mathrm{Support}(P_N)$ is equal to $M_N$, by Lemma \ref{unionbound} we have $\E_{\mathbf Z \sim \mathrm Q_{0,N}} \max_{X' \in \mathrm{Support}(P_N)}  \langle X',  \mathbf Z \rangle^2=O\left( \log M_N\right)$. Therefore since $\lambda_N=O(\log M_N)$,
\begin{align*}
 \E_{\mathbf Y \sim Q_{\lambda_N,N} }\max_{X' \in \mathrm{Support}(P_N)} \langle X', \mathbf Y \rangle ^2 \leq O\left( \lambda_N +\log M_N\right)=O\left( \log M_N\right).
\end{align*} Combining the last inequality with \eqref{eq:quadratic_ineq}, we conclude that
\begin{equation*} \E_{\mathbf Y \sim \mathrm Q_{\lambda_N,N} } \log^2 Z(\mathbf Y) = O\left(\lambda^2_N\right)=O\left(\log^2 M_N\right).
\end{equation*} Using \eqref{eq:banks} completes the proof of the proposition.
\end{proof}

\begin{proof}[Proof of Proposition~\ref{prop:mn_bound}]
We let $C$ denote an absolute positive constant whose value may change from line to line.
Let us write $\mathbf W = \langle  X, \mathbf Z \rangle/\sqrt{\lambda_N}$ and $\mathbf W' = \langle  X', \mathbf Z \rangle/\sqrt{\lambda_N}$. Recall that $X,X'$ lie on the sphere with $\langle  X, X' \rangle = \rho$.

Then $\mathbf W$ and $\mathbf W'$ are are jointly Gaussian with mean $0$ and covariance $\frac{1}{\lambda_N} \begin{pmatrix}1 & \rho \\ \rho & 1\end{pmatrix} =: \frac{1}{\lambda_N} \Sigma_\rho$.
Under this parametrization, we have
\begin{equation*}
\exp(\sqrt{\lambda_N}(\langle  X, \mathbf Z \rangle + \langle X', \mathbf Z \rangle) - \lambda_N) = \exp(\lambda_N(\mathbf W + \mathbf W' - 1))\,.
\end{equation*}
Let us write $S$ for the set $\{(w, w') : |w - 1| \leq \lambda_N^{-1/4}, |w' - 1| \leq \lambda_N^{-1/4}\}$.

We consider three cases:
\paragraph{Case 1: $\rho \leq 0$}
Using the moment generating function of the univariate normal distribution yields
\begin{equation*}
\E \exp(\lambda_N(\mathbf W + \mathbf W' - 1)) \1_{S}(\mathbf W, \mathbf W') \leq \E \exp(\lambda_N(\mathbf W + \mathbf W' - 1)) = e^{\lambda_N \rho} \leq 1\,,
\end{equation*}
so
\begin{equation*}
\frac{1}{\lambda_N} \log m_N(\rho) \leq 0 = \left(\frac{\rho}{1+\rho}\right)_+\,.
\end{equation*}

\paragraph{Case 2: $\rho \in (0, 1/2]$}
Write $\phi_\rho(w, w')$ for the joint density of $\mathbf W$ and $\mathbf W'$.
Note that on $S$
\begin{align*}
\phi_{\rho}(w, w') & \leq \frac{\lambda_N}{2 \pi (1-\rho^2)} \exp\left(- \frac{\lambda_N}{2} \mathbf w^\top \Sigma_\rho^{-1} \mathbf w\right)\,, \quad \quad \mathbf w = (w, w') \\
& \leq C e^{- \frac{\lambda_N}{1+\rho} + C \lambda_N^{3/4}},\,
\end{align*}where we use that $\lambda_N \rightarrow +\infty$ as $N \rightarrow +\infty$.
Hence
\begin{align*}
\frac{1}{\lambda_N} \log m_N(\rho) & = \frac{1}{\lambda_N} \log \int_S e^{\lambda_N(w + w' - 1)} \phi_\rho(w, w') \dd w \dd w' \\
& \leq \frac{1}{\lambda_N} \log \int_S \max_{(w, w') \in S} e^{\lambda_N(w + w' - 1)} \cdot \max_{(w, w') \in S}\phi_\rho(w, w') \dd w \dd w' \\
& \leq \frac{1}{\lambda_N} \log (\mathrm{vol}(S) \cdot e^{\lambda_N + O(\lambda_N^{3/4})} \cdot C e^{- \frac{\lambda_N}{1+\rho} + C \lambda_N^{3/4}})\\
& \leq \frac{\rho}{1+\rho} + \frac{C}{\lambda_N^{1/4}}\,.
\end{align*}

\paragraph{Case 3: $\rho \in (1/2, 1]$}
The sum $\mathbf W + \mathbf W'$ is Gaussian with mean $0$ and variance $\frac{2}{\lambda_N} (1 + \rho)$, and if $(w, w') \in S$, then $|w + w' - 2| \leq 2 \lambda_N^{-1/4}$.

We obtain
\begin{equation*}
m_N(\rho) = \E \exp(\lambda_N(\mathbf W + \mathbf W' - 1)) \1_{S}(\mathbf W, \mathbf W') \leq \E \exp(\lambda_N(\mathbf W'' - 1)) \1_{|\mathbf W'' - 2| \leq 2\lambda_N^{-1/4}}\,,
\end{equation*}
where $\mathbf W'' \sim \cN(0, \frac{2}{\lambda_N} (1 + \rho))$.
Similar with the analysis in 
Case 2, the density of $\mathbf W''$ is bounded by $C e^{- \frac{\lambda_N}{1 + \rho} + C \lambda_N^{3/4}}$ on the set $T \defeq \{w'': |w'' - 2| \leq 2\lambda_N^{-1/4}\}$, and we obtain
\begin{align*}
\frac{1}{\lambda_N} \log m_N(\rho) & \leq \frac{1}{\lambda_N} \log \int_{T} \max_{w'' \in T} e^{\lambda_N(w'' - 1)} \cdot C e^{- \frac{\lambda_N}{1 + \rho} + C \lambda_N^{3/4}} \\
& \leq \frac{1}{\lambda_N} \log (\mathrm{vol}(T) \cdot e^{\lambda_N + O(\lambda_N^{3/4})} \cdot C e^{- \frac{\lambda_N}{1+\rho} + C \lambda_N^{3/4}}) \\
& \leq \frac{\rho}{1+\rho} + \frac{C}{\lambda_N^{1/4}}\,,
\end{align*}
as claimed.
\end{proof}

\section{Additional lemmas}

\begin{lemma}\label{lem:mutual_information}
Denote by $I_{\lambda, N}(\mathbf X; \mathbf Y)$ the mutual information between $\mathbf X$ and $\mathbf Y$ in~\eqref{observation_model}, and denote by $\mathrm Q_{\lambda, N}^{(\mathbf X, \mathbf Y)}$ their joint law.
Then
\begin{equation*}
I_{\lambda, N}(\mathbf X; \mathbf Y) = \kl{\mathrm Q_{\lambda, N}^{(\mathbf X, \mathbf Y)}}{\mathrm P_N \otimes \mathrm Q_{\lambda, N}} = \frac{\lambda}{2} - \kl{\mathrm Q_{\lambda, N}}{\mathrm Q_{0, N}}\,.
\end{equation*}
\end{lemma}
\begin{proof}
The first equality is the definition of mutual information.
We then have
\begin{align*}
\kl{\mathrm Q_{\lambda, N}^{(\mathbf X, \mathbf Y)}}{\mathrm P_N \otimes \mathrm Q_{\lambda, N}} &= \E_{\mathrm Q_{\lambda, N}^{(\mathbf X, \mathbf Y)}} \log \frac{\mathrm Q_{\lambda, N} \left(\mathbf Y| \mathbf X\right) }{\mathrm Q_{\lambda, N} \left(\mathbf Y\right)} \\
&= \E_{\mathrm Q_{\lambda, N}^{(\mathbf X, \mathbf Y)}} \log \frac{\mathrm Q_{\lambda, N} \left(\mathbf Y|\mathbf X\right) }{\mathrm Q_{0, N} \left(\mathbf Y\right)}- \E_{\mathrm Q_{\lambda, N}} \log \frac{\mathrm Q_{\lambda, N} \left(\mathbf Y\right)}{\mathrm Q_{0, N} \left(\mathbf Y\right) }\,.
\end{align*}
Using the fact that $\mathbf Z$ has i.i.d. standard Gaussian entries we have
$$\E_{\mathrm Q_{\lambda, N}^{(\mathbf X, \mathbf Y)}} \log \frac{\mathrm Q_{\lambda,N} \left(\mathbf Y|\mathbf X\right) }{\mathrm Q_{0} \left( \mathbf Y\right)} =\E_{\mathrm Q_{\lambda, N}^{(\mathbf X, \mathbf Y)}} \frac{ \|\mathbf Y\|^2_2-\|\mathbf Y-\sqrt{\lambda}\mathbf X\|^2_2}{2}=\frac{\lambda}{2}\,,$$
and by definition
$$\kl{\mathrm Q_{\lambda, N}}{\mathrm Q_{0, N}}= \E_{\mathrm Q_{\lambda, N}} \log \frac{\mathrm Q_{\lambda, N} \left(\mathbf Y\right)}{\mathrm Q_{0, N} \left(\mathbf Y\right) }.$$
The claim follows.
\end{proof}

\begin{lemma}\label{lem:kl_facts}
For all $N$ and $\lambda > 0$, the function $\beta \mapsto \frac{1}{\lambda} \kl{\mathrm Q_{\beta \lambda, N}}{\mathrm Q_{0, N}}$ is nonnegative, nondecreasing, and $1/2$-Lipschitz.
\end{lemma}
\begin{proof}
Let us fix some $N$ and $\lambda$. The nonnegativity follows from the nonnegativity of the KL divergence. 
By Lemma~\ref{lem:mutual_information}, we have
\begin{equation*}
\frac{1}{\lambda} \kl{\mathrm Q_{\beta \lambda, N}}{\mathrm Q_{0, N}} = \frac{\beta}{2} - \frac{1}{\lambda} I_{\beta \lambda, N}(\mathbf X; \mathbf Y)\,.
\end{equation*}
Differentiating with respect to $\beta$ and using the I-MMSE theorem~\cite{GuoShaVer05} we conclude
\begin{align}\label{immse}
\frac{d}{d \beta} \frac{1}{\lambda} \kl{\mathrm Q_{\beta \lambda, N}}{\mathrm Q_{0, N}}=\frac{1}{2}-\frac{1}{2} \mmse{N}{\beta \lambda}. 
\end{align}
The results that $\beta \mapsto \frac{1}{\lambda} \kl{\mathbf Y_{\beta \lambda}}{\mathbf Z}$ is nondecreasing and $1/2$-Lipschitz follow directly from the fact that $\mmse{N}{\beta \lambda} \in [0,1].$
\end{proof}

\begin{lemma}\label{lem:lb}
For all $\lambda \geq 0$,
\begin{equation*}
\kl{\mathrm Q_{\lambda, N}}{\mathrm Q_{0, N}} \geq \frac{\lambda}{2} - \log M_N\,.
\end{equation*}
\end{lemma}
\begin{proof}
Writing explicitly the Kullback-Leibler divergence gives
\begin{align*}
\kl{\mathrm Q_{\lambda, N}}{\mathrm Q_{0, N}} & = \E \log \frac{1}{M_N} \sum_{X' \in \mathrm{Support}(\mathrm P_N)} \exp\left(\sqrt{\lambda} \langle \mathbf Y, X' \rangle - \frac{\lambda}{2}\right) \quad \quad \mathbf Y \sim \mathrm Q_{\lambda, N} \\
& \geq \E \log \frac{1}{M_N} \exp\left(\sqrt{\lambda} \langle \mathbf Z, \mathbf X \rangle + \frac{\lambda}{2}\right) \\
& = \E \left\{\sqrt{\lambda} \langle \mathbf Z, \mathbf X \rangle + \frac{\lambda}{2} - \log M_N\right\} = \frac{\lambda}{2} - \log M_N\,,
\end{align*}
where the inequality follows from writing $\mathbf Y = \sqrt{\lambda} \mathbf X + \mathbf Z$ and taking only the $X' = \mathbf X$ term in the sum.
\end{proof}

\begin{lemma}\label{binary_divergence}
Let $\alpha_1=(\alpha_1)_{N \in \mathbb{N}}$ and $\alpha_2=(\alpha_2)_{N \in \mathbb{N}}$ be two sequences in $[0, 1]$ such that $\alpha_1 = 1 - o(1)$ and $\alpha_2 = o(1)$ as $N \to \infty$, and let $\lambda_N$ be any sequence tending to infinity as $N \rightarrow +\infty$ such that $\frac{1}{\lambda_N} d(\alpha_1 \,\|\, \alpha_2)$ is bounded.
Then
\begin{equation*}
\limsup_{N \to \infty} \frac{1}{\lambda_N} d(\alpha_1 \,\|\, \alpha_2) = \limsup_{N \to \infty} \frac{1}{\lambda_N} \log \frac{1}{\alpha_2}\,.
\end{equation*}
\end{lemma}
\begin{proof}
The given asymptotics imply
\begin{equation*}
\lim_{N \to \infty} (1- \alpha_1) \log \frac{1 - \alpha_1}{1 - \alpha_2} = 0\,.
\end{equation*}
Moreover, since $\alpha_1 \log \alpha_1$ is bounded, we have
\begin{equation*}
\lim_{N \to \infty} \frac{1}{\lambda_N} \alpha_1 \log \alpha_1 = 0\,.
\end{equation*}
Combining these facts yields
\begin{align*}
\limsup_{N \to \infty} \frac{1}{\lambda_N} d(\alpha_1 \,\|\, \alpha_2) & = \limsup_{N \to \infty} \frac{1}{\lambda_N} \alpha_1 \log \frac{\alpha_1}{\alpha_2} + (1 - \alpha_1) \log \frac{1 - \alpha_1}{1 - \alpha_2} \\
& = \limsup_{N \to \infty} \frac{1}{\lambda_N} \alpha_1 \log \frac{1}{\alpha_2}\,.
\end{align*}
Since $\frac{1}{\lambda_N} d(\alpha_1 \,\|\, \alpha_2)$ is bounded, so is the sequence $\frac{1}{\lambda_N} \alpha_1 \log \frac{1}{\alpha_2}$, and since $\alpha_1$ is bounded away from $0$, this implies that $\frac{1}{\lambda_N} \log \frac{1}{\alpha_2}$ is bounded as well.
Using that $\lim_{N \to \infty} \alpha_1 = 1$ therefore yields the claim.
\end{proof}

\begin{lemma}\label{unionbound}
Let $M,N \in \mathbb{N}$ and let $S$ be a discrete subset of the $N$-dimensional unit sphere with cardinality $M$. Then for $G$ the law of the $N$-dimensional random variable $\mathbf Z$ with i.i.d. standard Gaussian coordinates it holds 
\begin{equation*}E_{\mathbf Z \sim G} \max_{X' \in S}  \langle X',  \mathbf Z \rangle^2  =O\left( \log M\right).\end{equation*}

\end{lemma}

\begin{proof}
It suffices to show that 
\begin{equation*}E_{\mathbf Z \sim G} \max_{X' \in S}  \langle X',  \mathbf Z \rangle^2 \1 \left( \max_{X' \in S} \langle X',  \mathbf Z \rangle^2  \geq 2 \log M \right) =O\left( 1\right).\end{equation*} or 
\begin{equation*}\int_{0}^{\infty} G\left(\max_{X' \in S}  \langle X',  \mathbf Z \rangle^2 \geq 2 \log M+t \right)\dd t =O\left( 1\right).\end{equation*} Using a union bound argument and the fact that for all $X' \in S$ the quantity $  \langle X',  \mathbf Z \rangle $ follows a standard Gaussian distribution, we have for all $t \geq 0$,
\begin{equation*}
G \left(\max_{X' \in S}  \langle X',  \mathbf Z \rangle^2 \geq 2 \log M+t \right) \leq M \exp \left(-\log M-\frac{t}{2}  \right)=\exp \left(-\frac{t}{2}  \right). \end{equation*}
Hence
\begin{equation*}\int_{0}^{\infty} G \left(\max_{X' \in S}  \langle X',  \mathbf Z \rangle^2 \geq 2 \log M+t \right)\dd t \leq \int_{0}^{\infty} \exp \left(-\frac{t}{2}\right) \dd t=O\left( 1\right),\end{equation*}as we wanted.

\end{proof}

\begin{lemma}\label{lem:hypergeom}
Suppose that $k=o(p)$ and the prior $\tilde{\mathrm P}_p$ is the uniform distribution on all the $k$-sparse vectors with elements either $0$ or $1/\sqrt{k}$. Then for any $t \in [0,1]$ it holds
\begin{equation*}
\limsup_{p \rightarrow + \infty}  \frac{1}{k \log \frac {p} {k}} \log \tilde{ \mathrm P}_p^{\otimes 2}[\langle \mathbf x, \mathbf x' \rangle \geq t] \leq -t.
\end{equation*}
\end{lemma}
\begin{proof} First note that the claim follows immediately when $t=1$ as $\tilde{\mathrm P}_p$ is a uniform distribution over a discrete space of cardinality $\binom{p}{k}$ and in our regime where $k=o(p)$ it holds $\log \binom{p}{k}=(1+o(1))k \log \frac {p} {k}$. Similarly, since for all $v,v'$ in the support of $\tilde{\mathrm P}_p$ it holds $\langle v,v' \rangle \geq 0$, the claim also follows straightforwardly for $t=0$. For the rest of the proof we assume $t \in (0,1)$.

The distribution of the rescaled overlap $k \langle \mathbf x,\mathbf x' \rangle=\langle \sqrt{k}\mathbf x,\sqrt{k}\mathbf x' \rangle$ follows the Hypergeometric distribution $\mathrm{Hyp}\left(p,k,k\right)$ with probability mass function $p(s)=\binom{k}{s}\binom{p-k}{k-s}/ \binom{p}{k}, s=0,1,\ldots,k.$
Therefore for a fixed $t \in (0,1]$,
\begin{equation}\label{Hyp:totalprob}
\tilde{\mathrm P}_p^{\otimes 2}[\langle \mathbf x, \mathbf x' \rangle \geq t]=\sum_{s=\ceil{tk}}^{k} p(s).
\end{equation}
Now for any $s \geq \ceil{tk}$ it holds
\begin{align*}
\frac{p(s+1)}{p(s)}=\frac{\binom{k}{s+1}}{\binom{k}{s}} \frac{\binom{p-k}{k-s-1}}{\binom{p-k}{k-s}} =\frac{(k-s)^2}{(s+1)(p-2k+s+1)}.
\end{align*} Using that $k=o(p)$ and $s \geq tk$ we conclude that for sufficiently large $p$ and all $s \geq \ceil{tk}$ it holds \begin{align*}
\frac{p(s+1)}{p(s)}\leq 2\frac{k}{tp}<\frac{1}{2}.
\end{align*} or by telescopic product,
\begin{align*}
\frac{p(s)}{p(\ceil{tk})}\leq \frac{1}{2^{s-\ceil{tk}}}.
\end{align*} Hence, using  \eqref{Hyp:totalprob} we have for large enough values of $p$,
\begin{equation}\label{Hyp:ineq1}
\tilde{\mathrm P}_p^{\otimes 2}[\langle \mathbf x, \mathbf x' \rangle \geq t] \leq \sum_{s=\ceil{tk}}^{k} p(\ceil{tk})\frac{1}{2^{s-\ceil{tk}}} \leq 2 p(\ceil{tk}).
\end{equation} 
Now using that $k<p/2$ for large enough $p$ we have
\begin{equation*}
p(\ceil{tk}) =\binom{k}{\ceil{tk}}\binom{p-k}{k-\ceil{tk}}/ \binom{p}{k} \leq 2^k \binom{p}{(1-t)k}/ \binom{p}{k}
\end{equation*} and combining with the elementary bounds $$m_2 \log \left(\frac{m_1}{m_2}\right) \leq \log \binom{m_1}{m_2} \leq m_2 \log \left(\frac{e m_1}{m_2}\right) =m_2\log \left(\frac{m_1}{m_2}\right)+O(m_2),$$ where  $m_1,m_2 \in \mathbb{N}$ with $m_1 \leq m_2$ we have 
\begin{align*}
\log p(\ceil{tk})  &\leq \left((1-t)k+1\right) \log \frac{p}{(1-t)k}-k\log \frac{p}{k} +O\left(k\right)\\
& = -tk \log \frac{p}{k} +(1-t)k \log \frac{k}{(1-t)k+1}+O\left(k\right),
\end{align*} where the last equality follows from elementary manipulations. Since $t $ is a fixed number in $(0,1)$ we have  
$(1-t)k \log \frac{k}{(1-t)k+1} \leq (1-t)k \log \frac{1}{1-t}=O\left(k\right)$ and therefore we conclude 
\begin{equation}\label{final_tk}
\log p(\ceil{tk}) \leq -tk \log \frac{p}{k} +O\left(k\right).
\end{equation} Combining \eqref{Hyp:ineq1} and \eqref{final_tk} and using the fact that $k=o(p)$ completes the proof.
\end{proof}

\end{document}